\documentclass{amsart}
\usepackage{amssymb, latexsym}
\usepackage{graphicx}
\usepackage{tikz-cd}
\usepackage{enumitem}
\newlist{steps}{enumerate}{1}
\setlist[steps, 1]{label = Step \arabic*:}

\newtheorem{theorem}{Theorem}
\newtheorem{remark}{Remark}
\newtheorem{lemma}{Lemma}
\newtheorem{definition}{Definition}

\newcounter{example}[section]
\newenvironment{example}[1][]{\refstepcounter{example}\par\medskip
   \noindent \textbf{Example~\theexample. #1} \rmfamily}{\medskip}

\usepackage{amsmath}

\begin{document}
\title{Berglund-H\"ubsch Transpose Rule and Sasakian Geometry}

\author{Ralph R. Gomez}
\address{ Department of Mathematics and Statistics\\
		Swarthmore College \\
		Swarthmore, PA 19081}

\date{\today}

\begin{abstract} We apply the Berglund-H\"ubsch transpose rule from BHK mirror symmetry to show that to an $n-1$-dimensional Calabi-Yau orbifold in weighted projective space defined by an invertible  polynomial, we can associate four (possibly) distinct Sasaki manifolds of dimension $2n+1$ which are $n-1$-connected and admit a metric of positive Ricci curvature. We apply this theorem to show that for a given K3 orbifold, there exists four seven dimensional Sasakian manifolds of positive Ricci curvature, two of which are actually Sasaki-Einstein.

\end{abstract}

\maketitle

\section{Introduction}\label{S:intro}
\indent In this article, we use the Berglund-H\"ubsch transposition rule that is involved in Berglund-H\"ubsch-Krawitz (BHK) mirror symmetry to show that to a compact $n-1$-dimensional Calabi-Yau hypersurface defined by an invertible polynomial in weighted projective space one can associate four (possibly) topologically distinct $2n+1$ dimensional Sasaki manifolds $M_i, i=1,\ldots,4$  of positive Ricci curvature for which $H_{n}(M_i, \mathbb{Z})$ is completely determined. Furthermore, we give an explicit example in which some of these manifolds are Sasaki-Einstein of positive scalar curvature. As we discuss, the construction proposed is motivated by a new relationship we observe between two specific Fano K\"ahler-Einstein orbifolds used in \cite{BG2} to prove the existence of Sasaki-Einstein metrics on $S^{2}\times S^{3}$. The key insight we notice is that these two orbifolds are related to one another via the Berglund-H\"ubsch transpose rule, which is a key rule in the so-called BHK construction of mirror Calabi-Yau pairs.\\
\indent 
Briefly, mirror symmetry is a conjectured phenomenon that occurs between Calabi-Yau varieties.
BHK mirror symmetry, in particular, is a construction developed by P. Berglund and T. H\"ubsch in the 1990s \cite{BerHub}, and later expanded upon and generalized by Krawitz in 2009 \cite{Kraw}. In their construction one begins by associating the $n\times n$ matrix of exponents from the invertible polynomial which defines a Calabi-Yau orbifold $X$ in weighted projective space. Then by taking the transpose of the matrix, called the Berglund-H\"ubsch transpose, it is possible to construct the associated polynomial of a certain degree from the transpose matrix, which generates another Calabi-Yau orbifold hypersurface $X^{\wedge}$ in a typically different weighted projective space. Then by taking certain quotients of $X$ and $X^{\wedge}$, one can create Calabi-Yau mirror pairs.\\
\indent Now, many examples of Sasaki manifolds of positive Ricci curvature (and Sasaki-Einstein) come from the the $S^{1}$-orbibundle construction developed in \cite{BG2} (see also \cite{BG6} for comprehensive review) which is of the form $S^{1}\rightarrow L_f\rightarrow X_f$ where $L_f$ is the link of an isolated hypersurface singularity defined by the quasi-homogeneous polynomials $f$ and $X_f$ is a Fano orbifold hypersurface in weighted projective space. The main goal of this paper is to apply the Berglund-H\"ubsch transpose rule to the Fano orbifold hypersurface $X_f$ and explore its implications for the Sasakian geometry and topology on $L_f.$\\
\indent One particularly interesting feature of the construction which we give as an example in Section 5 is that if one starts with a K3 orbifold mirror pair (in the sense of Dolgachev \cite{Dolg} and Voisin \cite{Vois}), then the construction can give rise to Sasaki-Einstein structures in dimension seven and therefore by the AdS/CFT correspondence there should be a related superconformal theory in dimension three. Thus the construction described in this article provides a framework by which one can go from a mirror symmetry setting to an AdS/CFT setting. 

\indent This paper is organized as follows: in Section 2 we review some preliminaries needed for the main theorem. In Section 3, we review the BHK mirror symmetry construction and in Section 4 we state and prove our main theorem. In Section 5 we present an application of the main theorem. 

\section*{Acknowledgements}
I would like to thank Marco Aldi for suggesting to me the idea of looking into how ideas from BHK mirror symmetry might be used in Sasakian geometry. I also would like to thank Christina T\o nneson-Friedman for useful conversations and encouragement in writing this article. Parts of this article were written while attending the Union College Mathematics Conference 2022 so I want to thank the organizers for creating such a hospitable place to share ideas. I would also like to thank Jaime Cuadros Valle and Tyler Kelly for useful conversations. Finally, I would like to that the Lang Faculty fellowship from Swarthmore College.

\section{Background }\label{B:Back}

\subsection{Links of Isolated Hypersurface Singularities}

Let us consider the affine space $\mathbb{C}^{n+1}$ in conjunction with a weighted $\mathbb{C}^{*}$ action, denoted by 
$$(z_{0},\dots,z_{n})\rightarrow (\lambda^{w_0}z_{0},\ldots ,\lambda^{w_n}z_n)$$ where $\lambda \in \mathbb{C}^{*}.$ Throughout this paper, we assume $gcd(w_{0},\ldots, w_{n})=1.$ The positive integral weights $w_i$ are viewed as the components of a weight vector $\mathbf{w}\in (\mathbb{Z}^{+})^{n+1}.$  We will also use the notation $|\mathbf{w}|$ to denote the sum of the weights in the weight vector. Now, let $f\in \mathbb{C}[z_0, \ldots , z_n]$. We say that $f$ is a quasi-homogeneous polynomial of degree $d$ if
$$f(\lambda^{w_0}z_0,\ldots, \lambda^{w_n}z_n)=\lambda^{d}f(z_0,\ldots, z_n)$$
such that $d\in \mathbb{Z}^{+}.$ Note that the terminology ``weighted homogeneous" is also used in the literature. The weighted affine cone denoted by $C_f$ is the zero set of $f$ and we will assume throughout the paper that the only singularity of $f$ is at the origin. We are interested in special manifolds that arise from this cone, called the link $L_f$, defined by
$$L_f=C_f\cap S^{2n+1}$$
where we view the odd dimensional unit sphere in $\mathbb{C}^{n+1}.$ It was proven by Milnor in \cite{Milnor}  that $L_f$ is a $(n-2)$-connected, compact smooth manifold of dimension $2n-1.$ It is now well-known that the link $L_f$ and weighted projective space $\mathbb{P}(\mathbf{w})$ are part of the following commutative diagram. See for example \cite{BG2} or \cite{BG6} for a review.

  \begin{center}
\begin{tikzcd}

L_f \arrow[r] \arrow[d, "\pi"] & S^{2n+1}_{\mathbf{w}} \arrow[d]\\
X_f \arrow[r] &\mathbb{P}(\mathbf{w})

\end{tikzcd}
\end{center}

\indent We think of the weighted projective space $\mathbb{P}(\mathbf{w})=\mathbb{C}^{n+1}\setminus \{0\}/ \mathbb{C}^{*}(\mathbf{w})$ where $\mathbb{C}^{*}(\mathbf{w})$ is the weighted $\mathbb{C}^{*}$ discussed above.\\
\indent Let us briefly explain the above commutative diagram. The manifold $S^{2n+1}_{\mathbf{w}}$ admits a weighted Sasakian structure and this induces an Sasakian structure on $L_f$.  A Sasakian structure on an odd-dimensional manifold is a quadruple of structure tensors $\mathcal{S}=(\xi, \eta, \Phi, g)$ where $g$ is a Riemannian metric, $\xi$ is a unit length Killing vector field called the Reeb vector field, $\eta $ a contact 1-form where $\xi$ is the associated Reeb vector field, the endomorphism of the tangent bundle $\Phi$ is a $(1,1)$ tensor field which annihilates $\xi.$ Moreover, $\Phi$, when restricted to the contact sub-bundle $\mathcal{D}=\text{ker} \eta$, is an integrable complex structure.
The odd-dimensional sphere $S^{2n+1}$ admits a so-called weighted Sasakian structure defined by inserting weights into the structure tensors. We do not need this explicitly so we will omit the precise formulas but the interested reader should consult \cite{BG6} for further details. \\
\indent The flow of the weighted Reeb vector field $\xi_{\mathbf{w}}$ generates a locally free $S^{1}$ action on $S^{2n+1}$ as well as $L_f.$ Taking quotients, we obtain the weighted hypersurface $X_f$ and $\mathbb{P}(\mathbf{w}).$ Note that $X_f$ is a compact K\"ahler orbifold which cuts out a hypersurface in $\mathbb{P}(\mathbf{w})$. Finally, the horizontal arrows in the commutative diagram above are Sasaki and K\"ahler embeddings while the vertical arrows comprise $S^{1}$ orbibundles and which are orbifold Riemannian submersions.\\
\indent  We now state a key theorem which provides the differential geometric relationship between links, positive Ricci curvature, and weighted hypersurfaces. 

\begin{theorem}\emph{(}\cite{Boyer1}\emph{)}
Let $L_f$ be a link of an isolated hypersurface singularity defined by a quasi-homogeneous polynomial $f$ of degree $d$ with a Sasakian structure so that we have the $S^{1}$ orbibundle $S^{1}\rightarrow L_f\rightarrow X_f.$ If $X_f$ is a Fano hypersurface of degree $d$ ($|\mathbf{w}|-d>0)$ in weighted projective space, then $L_f$ admits a Sasakian structure such that the Sasakian metric has positive Ricci curvature. Moreover, if $X_f$ is a Fano K\"ahler-Einstein orbifold hypersurface in $\mathbb{P}(\mathbf{w})$ such that  $$Id<\frac{n}{n-1}\emph{min}_{i,j}(w_{i}w_{j})$$
where $I=|\mathbf{w}|-d$ then $L_f$ is Sasaki-Einstein of positive scalar curvature.
\end{theorem}
Thus, links of isolated hypersurface singularities provides a setting for constructing Sasakian metrics of positive Ricci curvature as well as Sasaki-Einstein metrics on certain odd dimensional manifolds via the existence of special hypersurfaces in weighted projective space. It is through this setting of investigating hypersurfaces in weighted projective space that we turn now to a mirror symmetry construction which relies upon hypersurfaces in weighted projective space. This mirror symmetry construction is called BHK mirror symmetry.

\subsection{The BHK Mirror Symmetry Construction}
Mirror symmetry is a conjectured duality between Calabi-Yau varieties which was initially discovered by theoretical physicists Greene and Plesser \cite{GreenePlesser}. This duality originally arose in the context of string theory in the 1990s.  In its classical formulation the main idea of mirror symmetry is the following: if $X$ is a Calabi-Yau variety of dimension $n$, then there exists a Calabi-Yau variety $X^{\wedge}$ such that 
$$H^{p,q}(X,\mathbb{C})\cong H^{n-p,q}(X^{\wedge},\mathbb{C})$$
where $X$ and $X^{\wedge}$ are called a mirror pair. Since then there have been many mathematical constructions of mirror symmetry in a wide variety of contexts. Shortly after mirror symmetry emerged, the physicists P. Berglund and T. H\"ubsch  developed a very fascinating constructing that allowed one to construct many examples of mirror pairs \cite{BerHub}. This method was overshadowed by the powerful tools of toric constructions of mirror pairs developed by Batyrev and Borisov \cite{Bat}, \cite{BatBor}. A few years later, Krawitz \cite{Kraw} revived the construction by making it more mathematically precise and at the same time generalizing the construction. \\
\indent Consider the quasi-homogeneous polynomial $f$ of degree $d$ with weight vector $\mathbf{w}$ whose zero locus defines a hypersurface in weighted projective space $\mathbb{P}(\mathbf{w}).$ We shall consider a very special type of quasi-homogeneous polynomial, called an invertible polynomial, which can be represented by
\begin{equation}
f=\sum_{i=0}^{n}\prod_{j=0}^{n}x_j^{a_{ij}}.
\end{equation}
Note that the polynomial has $n$ variables containing $n$ monomials. We must also assume that $f$ when viewed as a map $f:\mathbb{A}^{n+1}\rightarrow \mathbb{A}^{1}$
has a single critical point only at the origin or equivalently its affine cone is smooth outside $(0,...,0).$ This condition is called quasi-smooth. This allows us to state the following:
\begin{definition}
A quasi-smooth, quasi-homogeneous polynomial, as given above by $(1)$ is invertible if the corresponding $n\times n$ matrix of exponents $A=(a_{ij})$ is invertible. 
\end{definition}

An invertible polynomial is a rather restrictive condition. We denote the inverse matrix by $A^{-1}=(a^{ij})_{i,j=1,...,n}.$ Observe that the matrix $A$, the weight vector $\mathbf{w}$ and the degree of the polynomial $d$ must satisfy the equation $A\mathbf{w}=\mathbf{d}$ where we view $\mathbf{d}$ as the $n\times 1$ column vector with all of its entries the scalar $d.$ Now, we can define $q_{i}=\sum_{j=1}^{n}a^{ij}$ to mean the sum of the entries of the $i$-th row of the inverse matrix $A^{-1}.$ Therefore, if we let $d$ be the least common denominator of the $q_{i}$ then the components of the weight vector $w_i$ are given by $w_{i}=d q_i.$\\
\indent It was shown in \cite{KreSka} that one can classify these polynomials. Such polynomials come in three types, sometimes called the \emph{atomic types}.
\begin{theorem}
Any invertible polynomial is a sum of atomic types:
\begin{enumerate}
\item Fermat: $f=x^{a}$
\item Loop: $f=x_{0}^{a_0}x_{1}+x_{1}^{a_1}x_2+\ldots +x_{n-1}^{a_{n-1}}x_{n}+x_{n}^{a_n}x_0$
\item Chain: $f=x_{0}^{a_0}x_{1}+x_{1}^{a_1}x_2+\ldots +x_{n-1}^{a_{n-1}}x_{n}+x_{n}^{a_n}$
\end{enumerate}

\end{theorem}

\indent We can now state the BHK mirror symmetry construction. Let $X_f$ be a Calabi-Yau hypersurface in $\mathbb{P}(\mathbf{w})$ defined by the invertible polynomial $f$ of degree $d$ and form the square matrix $A_f.$ Let $T$ denote transpose and form the matrix $A^{T}$. We call this the Berglund-H\"ubsch transpose rule. It is now possible to create the new invertible polynomial $f_{T}$
and so as we discussed above one needs to find $\mathbf{w}_{T}$ and $d_T$ which obey $A^{T}\mathbf{w}_T=\mathbf{d}_T. $ The new quasi-homogenous polynomial $f_{T}$ of degree $d_T$ in the weighted projective space $\mathbb{P}(\mathbf{w}_T)$ defines the orbifold $X_{f_T}$ which is again a Calabi-Yau hypersurface (see next section for proof) where $A^{T}_{f}=A_{f_{T}}.$ We can represent this process via the diagram:

 \begin{center}
\begin{tikzcd}

f \arrow[r] \arrow[d] & f_{T}\\
A_{f} \arrow[r,"T"] &A_{f}^{T}\arrow[u]\\
\end{tikzcd}
\end{center}
\indent It is worthwhile to mention that if $f$ is an invertible polynomial then so is $f_{T}$. This follows from Theorem 1 in \cite{KreSka}.\\
\indent Now, one would hope that $X_f$ and $X_{f_T}$ are mirror Calabi-Yau pairs but unfortunately it is not true. The Fermat quintic is a good counterexample.  If $W$ is the Fermat quintic, then its mirror  turns out to be the Calabi-Yau variety obtained after resolving singularities of $W/(5\mathbb{Z})^{3}$. For there to be a genuine mirror relation coming from the above Calabi-Yau varieties,  one must quotient each Calabi-Yau hypersurface by specially crafted groups which are related to the automorphsims of $f$ and $f_{T}.$ In fact, in \cite{ChiodoRuan},  Chiodo and Ruan were  able to prove a mirror symmetry relationship involving $X_f$ and $X_{f_T} $ which is stated as 
$$H^{p,q}_{CR}(X_{f}/ \tilde{G}, \mathbb{C}) \cong H^{n-1-p,q}_{CR}(X_{f_{T}}/\tilde{G}^{T},\mathbb{C})$$
with respect to Chen-Ruan cohomology. Here $\tilde{G}$ and $\tilde{G}^T$ are the carefully crafted groups coming from the automorphism groups. We do not need the precise formulation of these groups in this paper but their definitions can be found for example in  \cite{Kraw} and \cite{ChiodoRuan}.

\subsection{Sasakian Geometry and Topology of Links} 
To elucidate some of the topology of the link $L_f$, we can compute $b_{n-1}(L_f)$ and torsion in $H_{n-1}(L_{f}, \mathbb{Z}).$ This topological data is encoded in the degree $d$ and the weight vector $\mathbf{w}$ \cite{Orlik}. In this section we review the combinatorial formulas for the Betti numbers and the torsion groups in homology.\\
\indent Given $f$ of degree $d$ with weight vector $\mathbf{w}=(w_0, w_1,..., w_n)$, it is useful to first define the quantities
$$u_{i}=\frac{d}{gcd(d,w_{i})}, \hspace{1.5cm} v_{i}=\frac{w_{i}}{gcd(d,w_i)}.$$
The formula for the Betti number $b_{n-1}(L_f)$  was devised by Milnor and Orlik in \cite{MilnOrl} and is given by:
$$b_{n-1}(L_{f})=\sum(-1)^{n+1-s}\frac{u_{i_{1}},\ldots u_{i_s}}{v_{i_1}\ldots v_{i_s}\text{lcm}(u_{i_1},\ldots ,u_{i_s})}.$$
Here the sum is over all possible $2^{n+1}$ subsets $\{i_1,\ldots,i_s\}$ of $\{0,\ldots n\}.$\\
\indent For the torsion of the homology group, Orlik conjectured \cite{Orlik} that for a given link $L_{f}$ of dimension $2n-1$ one has
\begin{equation}
H_{n-1}(L_{f},\mathbb{Z})_{tor}=\mathbb{Z}_{d_1}\oplus\mathbb{Z}_{d_2}\oplus \cdots \oplus \mathbb{Z}_{d_r}
\end{equation}
Let us now discuss how the $d_{i}$ data are given, following closely \cite{BG6}.
Given an index set $\{i_1,i_2,....,i_s\}$, define $I$ to be
the set of all of the $2^s$ subsets and let us designate $J$ to be all of the proper subsets.
For each possible subset, inductively define the numbers $c_{i_{1},...,i_{s}}$
and $k_{i_{1},...,i_{s}}$. For each ordered
subset $\{i_{1},...,i_{s}\}\subset \{0,1,2,...,n\}$ with $i_{1}<i_2<\cdots <i_{s}$
one defines the set of $2^s$ positive integers, beginning with
$c_{\emptyset}=gcd(u_0,...,u_n):$
$$c_{i_{1},...,i_{s}}=\frac{gcd(u_{0},\ldots,\widehat{u}_{i_1},\ldots,\widehat{u}_{i_s},\ldots,u_n)}{\displaystyle \prod_{J}c_{j_{1},\ldots,j_{t}}}.$$
Now, to get the $k's$:
$$k_{i_{1},...,i_{s}}=\epsilon_{n-s+1}\kappa_{i_1,...,i_s}=\epsilon_{n-s+1}\displaystyle\sum_{I}(-1)^{s-t}\frac{u_{j_1}\cdots u_{j_t}}{v_{j_{1}}\cdots v_{j_t}lcm(u_{j_1},\ldots, u_{j_t})}$$
where
$$\epsilon_{n-s+1}=
\begin{cases}
0, & \text{if } n-s+1\text{ is even}\\
1, & \text{if } n-s+1\text{ is odd.}
\end{cases}
$$
Then for each $1\leq j \leq r=\lfloor max\{k_{i_{1},\ldots ,{i_s}}\}\rfloor$ we put
$$d_{j}=\prod_{k_{i_1,\ldots, i_s}\geq j}c_{i_{1},\ldots, i_{s}}.$$

Though the full conjecture is still open 45 years later, it is known to hold in certain cases.  However, substantial progress was made on the conjecture by Hertling and Mase in \cite{HertMase}. In particular, the authors were able to show the following:

\begin{theorem}
\emph{(}\cite{HertMase}\emph{)} Orlik's conjecture holds for all invertible polynomials.
\end{theorem}
Actually far more was proved in \cite{HertMase} but for our purposes, this form of the statement is all that is needed for the paper. Since we are dealing exclusively with invertible polynomials, this theorem allows to use the above formulas to explicitly compute $H_{n-1}(L_f, \mathbb{Z}).$

\section{The Berglund-H\"ubsch Transpose Rule and Sasakian Geometry}

\indent Our goal for this paper is to explore the Berglund-H\"ubsch transpose rule in the context of Fano  orbifolds in weighted projective space and thus generate Sasakian structures of positive Ricci curvature on the links. To emphasize the utility of using the rule for links in Sasakian geometry, we will discuss a novel interpretation of some examples in \cite{BG2}, which yielded new examples of Sasaki-Einstein metrics of positive scalar curvature in dimension five. 

\subsection{Key Example with Fano K\"ahler-Einstein Orbifolds}
In their paper \cite{BG2}, the authors show the existence of Sasaki-Einstein metrics on $S^{2}\times S^{3}$ and $2\#S^{2}\times S^{3}$. This is accomplished by using three Fano K\"ahler-Einstein orbifold hypersurfaces in weighted projective space and then arguing that the links over these hypersurfaces are diffeomorphic to the desired manifolds. These examples are given by quasi-homogeneous polynomials of degree 60 and 256.
\begin{enumerate}
\item $f_{60}=z_{0}^{5}z_{1}+z_{0}z_{2}^{3}+z_{1}^{4}+z_{3}^{3} \subset \mathbb{P}(9,15,17,20)$,
 \item $f_{256}=z_{0}^{17}z_{2}+z_{0}z_{1}^{5}+z_{1}z_{2}^{3}+z_{3}^{2}\subset \mathbb{P}(11,49,69,128)$, 
 \item $g_{256}=z_{0}^{17}z_1+z_{1}^{5}z_{2}+z_{0}z_{2}^{3}+z_{3}^{2} \subset \mathbb{P}(13,35,81,128).$
\end{enumerate}
\indent The hypersurfaces (and hence the links) are said to be \emph{well-formed}  if the hypersurface contains no codimension $2$ singular stratum of $\mathbb{P}(w_0, \ldots, w_n).$ It turns out that this condition is equivalent to 
$$\text{gcd}(w_{1},\ldots,\hat{w_i},\ldots \hat{w_j},\ldots w_n)| d$$
for all $i,j=1,\ldots n.$ The hat means delete the element. We will see later that interesting links also arise when the hypersurface is not well-formed.\\
\indent Let us focus on the second hypersurface $f_{256}$ in the weighted projective space $\mathbb{P}(11,49,69,128).$ Observe the quasi-homogenous polynomial is indeed invertible since it is quasi-smooth and in fact a sum of a loop and Fermat polynomial. We can then build the associated $4\times 4$ matrix $A$ given by
$$
A=
 \begin{bmatrix}
 17& 0 & 1 & 0\\
 1& 5 & 0 & 0\\
 0& 1 & 3 & 0\\
 0 & 0 & 0 & 2
  \end{bmatrix}
.$$ 
We compute now $A^{T}$ and so the corresponding quasi-homogenous polynomial is given by 
$$z_{0}^{17}z_1+z_{1}^{5}z_{2}+z_{0}z_{2}^{3}+z_{3}^{2}.$$
The observation to make here is that this polynomial is exactly $g_{256}$. Furthermore, to find the corresponding weight vector $\mathbf{w}$ and the degree we must solve $A^{T}\mathbf{w}=\mathbf{d}$ and we follow the procedure discussed in Section 2.2.  Since $A^{T}$ is invertible, to ensure $\mathbf{w}\in (\mathbb{Z}^{+})^{4}$ we get that $d=256$ and $\mathbf{w}=(13,35,81,128).$ We therefore have discovered the  Fano K\"ahler-Einstein oribifolds defined by $f_{256}$ and $g_{256}$ are actually related via the Berglund-H\"ubsch transpose rule. This example reveals that the Fano K\"ahler-Einstein condition in this example was preserved under the Berglund-H\"ubsch transpose rule.\\
\indent Let's interpret this in terms of the Sasaki-Einstein structures on the links over the orbifolds. As described in \cite{BG2}, it was further proved that the links over these two orbifolds are diffeomorphic to $S^{2}\times S^{3}.$ Thus, the Berglund-H\"ubsch transpose rule has allowed us, in this case, to transform from one Sasaki-Einstein structure on $S^{2}\times S^{3}$ to another distinct Sasaki-Einstein structure on $S^{2}\times S^{3}.$

\subsection{The invertible polynomial $f_{60}$}
One may ask what happens if we use the Berglund-H\"ubsch rule for the remaining polynomial $f_{60}.$

 In \cite{BG2}, the authors studied the following hypersurface given  by
 $$f_{60}=z_{0}^{5}z_{1}+z_{0}z_{2}^{3}+z_{1}^{4}+z_{3}^{3}$$
 of degree $d=60$ in the weighted projective space $\mathbf{w}=(9,15,17,20).$ It is known that this Fano orbifold admits an orbifold K\"ahler-Einstein metric \cite{DemKoll}.  By a change of variables, it is easy to see $f_{60} $ is a sum of a chain polynomial and Fermat monomial and so is invertible.  Now let's use the Berglund-H\"ubsch transpose rule. The exponential matrix for $f$ is given by
$$
A=
 \begin{bmatrix}
 5& 1 & 0 & 0\\
 1 & 0 & 3 & 0\\
 0& 4 & 0 & 0\\
 0 & 0 & 0 & 3
  \end{bmatrix}
$$ 
 Using the Bergl\"und-H\"ubsch transpose rule, we compute $A^{T}$. The corresponding invertible polynomial is therefore
 $$\tilde{f}_{T}=z_{0}^{5}z_{1}+z_{0}z_{2}^{4}+z_{1}^{3}+z_{3}^3.$$
 (We can rewrite this as ($z_{0}^{3}+z_{0}z_{1}^{5}+z_{1}z_{2}^{4})+z_{3}^{3}$ which is again a sum of a chain and Fermat.) 
It is easy to determine $\mathbf{w}_{T}$ and $d$ which obey $A\mathbf{w}_{T}=\mathbf{d}_T.$ We find that $\mathbf{w}_T=(8,20,13,20)$ and $d_{T}=d=60$ works. It is interesting to note that the corresponding link $L_{\tilde{f}}$ admits a Sasaki-Einstein metric since the inequality $Id<\frac{n}{n-1}\text{min}_{i,j}(w_{i}w_{j})$ is satisfied. Furthermore, it should be noted that the hypersurface does have a branch divisor for $C=\{z_{0}=0\}$ in the singular locus. By Theorem 5.7 of  \cite{KollarfiveD}, the genus of this branch divisor determines the topology of $L_f.$ Using the genus formula given by
 $$g(C)=\frac{1}{2}\left(\frac{d^{2}}{w_{1}w_{2}w_{3}}-d\sum_{i<j}\frac{\text{gcd}(w_i,w_j)}{w_{i}w_{j}}+\sum\frac{\text{gcd}(d,w_i)}{w_i}-1\right)$$
we find $g(C)=0$ which implies that $H_{2}(L_f, \mathbb{Z})$ has no torsion. Moreover, calculation of the second Betti number as discussed in  Section 2.3, we find $b_{2}(L_{\tilde{f}})=2$ and so $L_f$ is diffeomorphic to $2\# S^{2} \times S^{3}.$  (Alternatively, we could have also used the ideas from Section 2.3.) The Milnor number of the link is $94$ and is therefore distinct from other known Sasaki-Einstein structures on $2\#S^{2}\times S^{3}$.\\
 \indent These examples emphasize and highlight the utility of the Berglund-H\"ubsch transpose rule in relation to Sasakian geometry. In the next section, we establish the general connection between the two ideas.

\section{Main results}
In this section, we prove our main theorem. But first, we establish a simple lemma. The Calabi-Yau case is already well-known but we include a proof here for reference.

 \begin{lemma}
 Let $f$ be an invertible polynomial defining a hypersurface $X_{f}$ in weighted projective space $\mathbb{P}(\mathbf{w})$ and let $A=(a_{ij})_{i,j=1,...,n}$ be the associated invertible $n\times n$ matrix of exponents. Further let $X_{f_T}$ be the associated hypersurface in $\mathbb{P}(\mathbf{w}_{T})$ defined by $f_{T}$ of degree $d_{T}.$ Then we have:
 
\begin{enumerate}
\item If $X_{f}$ is a Calabi-Yau hypersurface, then so is $X_{f_{T}}.$
\item If $X_{f}$ is a Fano hypersurface, then so is $X_{f_{T}}$
 \end{enumerate}
 \end{lemma}

  \begin{proof}
 First, note that in either case, if $f$ is invertible then so is $f_{T}$ and this follows from Theorem 1 of \cite{KreSka}.
 Since $f$ is quasi-homogenous of degree $d$, we must have that $A\mathbf{w}=\mathbf{d}$ where $\mathbf{d}$ is the $n\times 1$ column vector with entries $d$.  The matrix $A$ is invertible with entries of  $A^{-1}$ denoted by $a^{ij}.$ Then the components of $\mathbf{w}$ which are positive integers denoted by $w_{i}$ are given by  $w_{i}=\sum_{j}a^{ij}d$. Moreover, since $X_f$ is Calabi-Yau, the condition $|\mathbf{w}|=d$ can be rewritten as 
 $$\sum_{i,j}a^{ij}=1.$$
 Now let $A^{T}$ denote the transpose of the exponential matrix. From this matrix, we would like to construct an associated invertible polynomial $f_{T}$  of degree $d_{T}$ in weighted projective space $\mathbb{P}(\mathbf{w}_T)$. To find the weights for $f_{T}$ we must solve $A^{T}\mathbf{w}_{T}=\mathbf{d}_{T}.$ Since $A$ is invertible we have
 $$\mathbf{w}_{T}=(A^{T})^{-1}\mathbf{d}_T=(A^{-1})^{T}\mathbf{d}_{T}$$ where $\mathbf{d}_{T}$ is an $n\times 1$ column vector with entries $d_{T}$. The components of $\mathbf{w}_{T}$,  which we will denote by $w^{T}_{i}$ to avoid cluttering of indices, are given by $$w^{T}_{i}=\sum_{i}a^{ji}d_{T}.$$   So to establish the first statement that $X_{f_T}$ is Calabi-Yau, we compute using the hypothesis $X_{f}$ is Calabi-Yau to obtain
 $$|\mathbf{w}_{T}|=\sum_{i,j}a^{ji}d_{T}=d_{T}.$$ Hence $|\mathbf{w}_T|=d_{T}$ and $X_{f_{T}}$ is Calabi-Yau.
 
 To establish the second statement, assume $X_{f}$ is Fano and hence the Fano condition is $|\mathbf{w}|-d>0$ i.e. 
 $$\displaystyle \sum_{i,j}a^{ij}>1.$$ 
 So, to establish that $X_{f_{T}}$ is Fano, we compute 
 $$|\mathbf{w}_{T}|-d_{T}=\left( \sum_{i,j}a^{ij}-1\right)d_{T}.$$ Since $\sum_{i,j}a^{ij}-1>0$ this implies $|\mathbf{w}_{T}|-d_{T}>0.$ Therefore, $X_{f_{T}}$ is Fano.
 
 \end{proof}

 We are now ready to state and prove the main theorem of this article.

\begin{theorem}
 Let $X_{f}$ be an $n-1$ dimensional Calabi-Yau hypersurface defined by an invertible polynomial $f(x_{0},...,x_{n})$, not of Fermat type, in weighted projective space $\mathbb{P}(w_{0},...,w_{n})$. Then there exists four (possibly) distinct $n-1$-connected Sasaki manifolds $M_{i}$, $i=1,2,3,4$ of dimension $2n+1$ of positive Ricci curvature such that $H_{n}(M_{i}, \mathbb{Z})$ is completely determined.
 \end{theorem}

 \begin{proof}
 Let $f$ be an invertible polynomial of degree in $d$ that cuts out a Calabi-Yau orbifold hypersurface $X_{f}$ in weighted projective space $\mathbb{P}(\mathbf{w}).$ Define the new quasi-homogenous polynomial of the same degree $d$ given by $\tilde{f}=z^{d}+f$ which generates the new hypersurface $Y_{\tilde{f}}$ in the weighted projective space $\mathbb{P}(1, \mathbf{w}).$  It is clear that $\tilde{f}$ is invertible since it is a sum of a Fermat monomial $z^d$ and the invertible polynomial $f.$ Note that $1+|\mathbf{w}|-d=1+d-d=1>0$, where we have used $|\mathbf{w}|=d$. Thus  $Y_{\tilde{f}}$ is an orbifold Fano hypersurface. By Lemma 1 above we can construct $Y_{\tilde{f}_{T}}$ which is also an orbifold Fano hypersurface. Now by Theorem 1, we have $S^{1}$-orbibundles $M_{1}$ and $M_{2}$ over $Y_{\tilde{f}}$ and $Y_{\tilde{f}_{T}}$ respectively. Let us represent what we have so far in a diagram:
 \begin{center}
\begin{tikzcd}
& M_{1}\arrow[d] & M_{2}\arrow [d]\\
X_f \arrow[r,"\theta"]  & Y_{\tilde{f}} \arrow[r, "T"]
& Y_{\tilde{f}_{T}}. \\
\end{tikzcd}
\end{center}
 
 We are abusing notation here slightly for the sake of illustrating how the different spaces are related to one another by the construction. The map $\theta$ takes the given Calabi-Yau hypersurface and assigns to it the Fano orbifold $Y_{\tilde{f}}$ by mapping $f$ to $\tilde{f}=z^{d}+f$ where $z^{d}$ is of weight $1.$ The $T$ maps one Fano hypersurface to another via the Berglund-H\"ubsch transpose rule.\\
\indent Now let's return to our initial Calab-Yau hypersurface $X_f.$ We can apply the Berglund-H\"ubsch transpose rule  to it obtaining $X_{f_T}$ where $f_T$ is of degree $d'$. Using the same process as discussed above, we can use the $\theta$ map and the Berglund-H\"ubsch transpose rule $T$ to obtain the following sequence:
 \begin{center}
\begin{tikzcd}
& M_{3}\arrow[d] & M_{4}\arrow [d]\\
X_{f_T} \arrow[r,"\theta"]  & Y_{\tilde{g}}\arrow[r, "T"]
& Y_{\tilde{g_{T}}} \\
\end{tikzcd}
\end{center}
 \noindent where we view $M_3$ and $M_4$ as the total spaces arising from Theorem 1. The polynomial $\tilde{g}=z^{d'}+f_{T}$ is the orbifold Fano hypersurface of degree $d'$ in some weighted projective space We can put the two diagrams together getting a summary of the entire process:
 \begin{center}
\begin{tikzcd}
& M_{1}\arrow[d] & {M_2}\arrow [d]\\
X_f \arrow[r,"\theta"] \arrow[d, "T"] & Y_{\tilde{f}} \arrow[r, "T"]
& Y_{\tilde{f}_{T}} \\
X_{f_T} \arrow[r, "\theta"] &Y_{\tilde{g}}\arrow[r, "T"]
& Y_{\tilde{g}_{T}}\\
&M_{3}\arrow[u] & M_{4}\arrow [u]
\end{tikzcd}
\end{center}
 To conclude the proof that $M_i, i=1,2,3,4$ are Sasakian manifolds admitting metrics of positive Ricci curvature such that $H_{n}(M_{i}, \mathbb{Z})$ can be explicitly determined we now apply Theorem 1 and Theorem 3. The fact that $M_i$ are $n-1$-connected follows from the Milnor fibration theorem. This completes the proof.
  \end{proof}
  \begin{remark}
  In the hypothesis of Theorem 1 we exclude the case in which the initial Calabi-Yau orbifold defined by the invertible polynomial $f$ is Fermat since in this case $A$ is diagonal and so $A^{T}=A$ and so one gets the same Calabi-Yau orbifold upon application of the Berglund-H\"ubsch rule.\\
  \end{remark}
  \indent We have shown that if one starts with an input Calabi-Yau orbifold $X_f$ defined by an invertible polynomial, then it is possible to construct four Sasakian manifolds $M_{i}$ of positive Ricci curvature from it.\\

\section{Application}
In this section, we apply the main theorem to the case in which the initial data is a K3 orbifold hypersurface coming from Reid's list \cite{ReidsList} (see also \cite{Iano-Fletcher}). By Theorem 4 this will yield four topologically distinct Sasakian 7-manifolds of positive Ricci curvature. However, we can use results of I. Cheltsov \cite{Chelt} to conclude that at least two of the four Sasaki 7-manifolds are actually Sasaki-Einstein. More precisely, in \cite{Chelt} it was shown that from Reid's famous list of orbifold K3 surfaces in weighted projective space $\mathbb{P}(w_0, w_1, w_2, w_3)$, the Fano hypersurfaces in $\mathbb{P}(1, w_0, w_1, w_2, w_3)$ are also K\"ahler-Einstein (with the exception of four cases.) These hypersurface Fano 3-folds are shown to be K\"ahler-Einstein by analyzing the so-called log canonical threshold.\\

\begin{example}
 Consider the $K3$ hypersurface $X_{f}=\{f=z_1^{2}z_4+z_{2}^{4}+z_{1}z_{3}^{3}+z_{4}^8=0\}$ of degree $d =16$ in weighted projective space $\mathbb{P}(7,4,3,2).$ Using the map $\theta$, we can introduce a Fermat term $z^{16}_{0}$ of weight one to obtain the invertible polynomial $\tilde{f}=z_{0}^{16}+f$ and so this gives a Fano orbifold hypersurface $Y_{\tilde{f}}$ in $\mathbb{P}(1,7,4,3,2).$ By Corollary 1.4 in \cite{Chelt}, $Y_{\tilde{f}}$ is a Fano K\"ahler-Einstein orbifold and thus by Theorem 1 is a Sasaki-Einstein $7$-manifold $M_{1}$. From the invertible polynomial $\tilde{f}$ we can associate its exponent matrix $A$  given by 
 $$A=
 \begin{bmatrix}
 16 & 0 & 0 & 0 & 0\\
 0 & 2 & 0 & 0 & 1\\
 0 & 0 & 4 & 0 & 0\\
 0 & 1 & 0 & 3 & 0\\
 0& 0 & 0 & 0 &8\\
  \end{bmatrix} . 
 $$
 We apply the Berglund-H\"ubsch transpose rule to obtain $A^T$ and thus determine the Fano orbifold $Y_{\tilde{f}_{T}}$ defined by $\tilde{f}_{T}=z_{0}^{16}+z_{1}^{2}z_{3}+z_{2}^{4}+z_{3}^{3}+z_{1}z_{4}^{8}.$ A straightforward calculation determines $\mathbf{w}=(3,16,12,16,4)$ and the degree $d_{\tilde{f}_T}=48.$ By Theorem 1, $M_{2}$ is indeed a Sasakian manifold of positive Ricci curvature. It is unknown if  $M_{2}$ is actually Sasaki-Einstein. The link $M_{2} $ does not satisfy any of the known obstructions \cite{GMSY} to admit a Sasaki-Einstein metric so it is still possible the link $M_2$ is Sasaki-Einstein. We have so far been unable to prove this. Note that $Y_{\tilde{f}_{T}}$ is not well-formed, i.e. has a branch divisor so that complicates the matter.\\
 \indent Now going back to our input K3 orbifold hypersurface $X_{f}$ we can compute the exponential matrix $B$ given by
 $$B
 =
  \begin{bmatrix}
 2 & 0 & 0 & 1 \\
 0 & 4 & 0 & 0\\
 1 & 0 & 3 & 0 \\
 0 & 0 & 0 & 8\\
  \end{bmatrix} . 
 $$
 This yields the $K3$ surface $X_{f_{T}}$ defined by $f_{T}=z_{1}^{2}z_{3}+z_{2}^{4}+z_{3}^{3}+z_{1}z_{4}^{8}$ of degree $d'=12$ with weight vector $\mathbf{w}=(4,3,4,1).$  As before we apply the $\theta$ to introduce the Fermat term $z_{0}^{12}$ of weight one to obtain $\tilde{g}=z_{0}^{12}+f_{T}$ 
 and this gives a Fano orbifold hypersurface $Y_{\tilde{g}}$ in $\mathbb{P}(1,4,3,4,1)$. Again,  by Corollary 1.4 in \cite{Chelt}, this hypersurface is Fano K\"ahler-Einstein and therefore there exists a Sasaki-Einstein $7$ manifold $M_{3}$ which fibers over the base $Y_{\tilde{g}}$ by Theorem 1.\\
\indent Furthermore, we can now apply the Berglund-H\"ubsch transpose rule on $\tilde{g}$ to obtain $\tilde{g}_T=\{z_{0}^{12}+z_{1}^2z_{4}+z_{2}^{4}+z_{1}z_{3}^{3}+z_{4}^{8}=0\}$  of degree $48$ with weight vector $\mathbf{w}=(4,21,12,9,6).$ Thus we obtain the Fano orbifold $Y_{\tilde{g}_{T}}$ and so by Theorem 1, there exists a Sasakian 7-manifold $M_{4}$ with positive Ricci curvature. It is not known whether this manifold is actually Sasaki-Einstein.\\
\indent Furthermore, we can explicitly compute $H_{3}(M_{i},\mathbb{Z}) $ for $i=1,2,3,4$ using the methods described in Section 2.3.  In fact, a computer program\footnote{https://homologyoflinks.fishtank.swarthmore.edu} based on Orlik's conjecture was written by Evan Thomas which allows us to quickly compute the homology groups of the links. Therefore, we see that to the initial $K3$ orbifold hypersurface $X_{f}$, we can associate four manifolds of dimension $7$ which are $2$-connected and of positive Ricci curvature. At least two of these manifolds, are actually Sasaki-Einstein. We summarize the result  in Table 1. \\

\begin{table}
\begin{tabular}{ |p{1.7cm}|p{3.7cm}|p{3.7cm}|  }
\hline
%\multicolumn{3}{|c|}{ Input CY orbifold $z^{2}_{1}z_{4}+z_{2}^4+z_{1}z_{3}^{3}+z_{4}^{8}$} \\
%\hline
 $M_i$ & polynomial & $H_{3}(M_{i},\mathbb{Z})$\\
\hline
$M_1$ & $z^{16}_{0}+z^{2}_{1}z_{4}+z_{2}^4+z_{1}z_{3}^{3}+z_{4}^8$  & $\mathbb{Z}^{108}\oplus\mathbb{Z}_{8}\oplus(\mathbb{Z}_2)^2$ \\ \hline

$M_2$ & $z^{16}_{0}+z^{2}_{1}z_{3}+z_{2}^4+z_{3}^{3}+z_1z_{4}^{8}$   &$\mathbb{Z}^{30}\oplus(\mathbb{Z}_{48})^2\oplus\mathbb{Z}_{12}\oplus (\mathbb{Z}_{4})^9$  \\ \hline

$M_3$ &$z_{0}^{12}+z^{2}_{1}z_{3}+z_{2}^{4}+z_{3}^{3}+z_{1}z_{4}^{8}$& $\mathbb{Z}^{120}\oplus(\mathbb{Z}_4)^2$ \\ \hline

$M_4$   &$z_{0}^{12}+z_{1}^{2}z_{4}+z_{2}^{4}+z_{1}z_{3}^{3}+z_{4}^8$&$\mathbb{Z}^{24}\oplus\mathbb{Z}_{24}\oplus(\mathbb{Z}_6)^2 \oplus (\mathbb{Z}_{3})^{6}$  \\

\hline
\end{tabular}
\caption{Sasaki manifolds of positive Ricci curvature from input CY}
\end{table}
\end{example}
 The above example also illustrates an interesting feature of the construction. The initial K3 orbifold $X_f$ is mirror dual to $X_{f_T}$ in the sense of Dolgachev and Voisin by \cite{BottCompPrid} Table 10. Now because of $M_1$ and $M_3$ are Sasaki-Einstein $7$-manifolds, we know that the AdS/CFT correspondence should give rise to a three dimensional superconformal theory \cite{GMSW}. Therefore, our construction illustrates a way to move from a mirror symmetry framework to an AdS/CFT framework.

\end{document}